\documentclass[11pt]{amsart}
\usepackage{enumerate}
\usepackage{amsmath}
\usepackage{amssymb}
\usepackage{graphicx}
\usepackage{cite}

\newtheorem{theorem}{Theorem}[section]
\newtheorem{corollary}[theorem]{Corollary}
\newtheorem{lemma}[theorem]{Lemma}
\newtheorem{proposition}[theorem]{Proposition}
\theoremstyle{definition}
\newtheorem{definition}[theorem]{Definition}
\newtheorem{example}[theorem]{Example}
\newtheorem{remark}[theorem]{Remark}

\numberwithin{equation}{section}

\title[The Hille-Yosida generation theorem for $C_{0}$--random semigroups]{The Hille-Yosida generation theorem for almost surely bounded $C_{0}$--semigroups of continuous module homomorphisms$^1$}

\author[X. Zhang]{Xia Zhang}

\address[X. Zhang]{School of Mathematical Sciences, TianGong University, Tianjin 300387, P.R.China}
\email{{\tt zhangxia@tiangong.edu.cn}}

\author[M. Liu]{Ming Liu}
\address[M. Liu]{School of Mathematical Sciences, TianGong University, Tianjin 300387, P.R.China}
\email{\tt liuming@tiangong.edu.cn}

\author[T. X. Guo]{Tiexin Guo$^*$}

\address[T. X. Guo]{School of Mathematics and Statistics, Central South University,
 Changsha 410083, P.R.China}
\email{{\tt tiexinguo@csu.edu.cn}}
\thanks{$^*$Corresponding author}



\begin{document}

\begin{abstract}
In this paper, we first study some properties peculiar to $C_{0}$--semigroups of continuous module homomorphisms and give a characterization for such a $C_{0}$--semigroup to be almost surely bounded. Then, based on these, we establish the Hille-Yosida generation theorem for almost surely bounded $C_{0}$--semigroups of continuous module homomorphisms, which generalizes some known results. Moreover, the counterexample constructed in this paper also shows that it is necessary to require the almost sure boundedness for such $C_{0}$--semigroups.
\end{abstract}

\maketitle


\section{Introduction}\label{section1}

Random functional analysis is functional analysis based on random metric spaces, random normed modules, random inner product modules and random locally convex modules, which are random generalizations of ordinary metric spaces, normed spaces, inner product spaces and locally convex spaces, respectively. Now, random functional analysis has undergone a deep and systematic development \cite{Guo93,Guo96a,Guo08,Guorelations,Guohomomorphism,GL05} and been successfully applied to
the study of conditional convex risk measures and backward stochastic equations \cite{Guoprogress,Guo2019,Guoconvex,Guozhang2019}. A historic survey of random functional analysis can be found in \cite{Guorelations,Guoconvex,Guozhang2019}.

The theory of semigroups of linear operators is an important part of functional analysis and applied in various fields \cite{Engel2000,Goldstein,Pazy}. A.V.Skorohod has widely studied the theory of semigroups of random linear operators in order to study random operator equations and random integral equations in \cite{Skorohod}. Motivated by the idea of \cite{Guo93}, in \cite{Guozhang} we gave stone's representation theorem of a group of random unitary operators on complete complex random inner product modules, which generalized the corresponding result for a group of random unitary operators on a separable Hilbert space in \cite{Skorohod}. In particular, we obtained the fundamental theorem of calculus for a Lipschitz function from a finite real interval to a complete $RN$ module, which has played a crucial role in \cite{Guozhang,Thang,ZhangLiu}. Thang et.al recently established the Hille--Yosida theorem for a contraction semigroup of continuous module homomorphisms on complex complete $RN$ modules. The central purpose of this paper is to establish the Hille--Yosida theorem for an almost surely bounded $C_0$--semigroup of continuous module homomorphisms on complex complete $RN$ modules, where the main difficulty is to provide some techniques giving some important properties peculiar to almost surely bounded semigroups of continuous module homomorphisms. The results of this paper generalizes and improves those in \cite{Guozhang,Thang,ZhangLiu}.

The remainder of this paper is organized as follows: in Section 2 we briefly recall some basic notions and facts; in Section 3 we give a characterization for a $C_{0}$--semigroup of continuous module homomorphisms to be almost surely bounded and in Section 4 on the basis of Section 3 we are devoted to establishing the Hille-Yosida generation theorem for such $C_{0}$--semigroups.

\section{Preliminaries}\label{section2}
Throughout this paper, $N$ denotes the set of positive integers, $K$ the scalar field $R$ of real numbers or $C$ of complex numbers, $(\Omega,\mathcal F,P)$
a given probability space, $\bar{L}^{0}(\mathcal F,R)$
the set of equivalence classes of extended real-valued $\mathcal F-$measurable random variables on $\Omega$, $L^{0}(\mathcal F,K)$
the algebra of equivalence classes of $K$-valued $\mathcal F-$
measurable random variables on $\Omega$ under the ordinary addition, scalar multiplication and multiplication operations on equivalence classes.

\begin{proposition}\rm(see \cite{Dunford})
$\bar{L}^{0}(\mathcal F,R)$ is a complete lattice under the ordering $\leq$: $\xi\leq \eta$ if and only if $\xi^{0}(\omega)\leq\eta^{0}(\omega)$, for almost all $\omega$ in $\Omega$ (briefly, a.s.), where $\xi^{0}$ and $\eta^{0}$ are arbitrarily chosen representatives of $\xi$ and $\eta$, respectively, and has the following three properties:

(1) For each subset $A$ of $\bar{L}^{0}(\mathcal F,R)$, there are two sequences $\{a_{n},~n\in N\}$ and ${\{b_{n},~n\in N}\}$ in $A$ such that $\bigvee_{n\geq1}$ $a_{n}=\bigvee A$ and $\bigwedge_{n\geq1}$ $b_{n}=\bigwedge A$;

(2) If $A$ is directed (dually directed), namely for any two elements $c_1$ and $c_2$ in $A$ there is some $c_3$ in $A$ such that $c_1 \bigvee c_2 \leq c_3 ~  (c_1 \bigvee c_2 \geq c_3)$, then the above $\{a_{n},~n\in N\}$ ($\{b_{n},~n\in N\}$) can be chosen as nondecreasing (nonincreasing);

(3) $L^{0}(\mathcal F,R)$, as a sublattice of $\bar{L}^{0}(\mathcal F,R)$, is conditionally complete.
\end{proposition}

Specifically, we denote $L^{0}_{+}(\mathcal {F})=\{\xi\in L^{0}(\mathcal {F},R)\,|\,\xi\geqslant 0\}$ and $L^{0}_{++}(\mathcal {F})=\{\xi\in L^{0}(\mathcal{F},R)\,|\,\xi> 0 ~\textrm{on}~  \Omega \}.$

\begin{definition}\rm(see \cite{Guo93,Guorelations})
An ordered pair $(S,\|\cdot\|)$ is called an $RN$ module over $K$ with base $(\Omega,{\mathcal F},P)$ if $S$ is a left module over the algebra $L^{0}({\mathcal F},K)$ and $\|\cdot\|$ is a mapping from $S$ to $L^{0}_{+}(\mathcal {F})$ such that the following three axioms are satisfied:

(1)\, $\|\xi x\|=|\xi|\cdot\|x\|,\forall\xi\in L^{0}({\mathcal F},K)$ and $x\in S$;

(2)\, $\|x+y\|\leq \|x\|+\|y\|,\forall x,y\in S$;

(3)\, $\|x\|=0$ implies $x=\theta$ (the null vector of $S$), where $\|\cdot\|$ is called the $L^{0}-$norm on $S$.
\end{definition}

Let $(S,\|\cdot\|)$ be an $RN$ module over $K$ with base $(\Omega,{\mathcal F},P)$. For any positive real numbers $\varepsilon$ and $\lambda$ such that $\lambda<1$, let $N_{\theta}(\varepsilon,\lambda)=\{x\in X~|~P\{\omega\in \Omega~|~\|x\|(\omega)<\varepsilon\}>1-\lambda\},$ then $\{N_{\theta}(\varepsilon,\lambda)~|~\varepsilon>0,~0<\lambda<1\}$ is a local base at the null vector $\theta$ of some Hausdorff linear topology, and the linear topology is called the $(\varepsilon,\lambda)-$ topology. It should be pointed out that the idea of introducing the $(\varepsilon,\lambda)-$ topology is due to Schweizer and Sklar \cite{SS}. In this paper, given an $RN$ module $(S,\|\cdot\|)$, it is always assumed that $(S,\|\cdot\|)$ is endowed with the $(\varepsilon,\lambda)-$topology. Besides, it is worth noting that a sequence $\{x_{n},n \in N\}$ in $S$ converges to $x\in S$ in the $(\varepsilon,\lambda)-$topology if and only if $\{\|x_{n}-x\|,n \in N\}$ converges to 0 in probability $P$.


\begin{definition}\rm(see \cite{Guohomomorphism})
Let $(S^{1},\|\cdot\|_{1})$ and $(S^{2},\|\cdot\|_{2})$ be two $RN$ modules over $K$ with base $(\Omega,{\mathcal F},P)$. A linear operator $T$ from $S^{1}$ to $S^{2}$ is called a generalized random linear operator (briefly, a random linear operator), and further the random linear operator $T$ is called almost surely bounded (briefly, a.s. bounded) if there exists a $\xi\in L^{0}_{+}(\mathcal {F})$ such that $\|Tx\|_{2}\leq \xi\cdot\|x\|_{1}$ for any $x\in S^{1}$. Denote by $B(S^{1},S^{2})$ the linear space of a.s. bounded random linear operators from $S^{1}$ to $S^{2}$, define $\|\cdot\|:B(S^{1},S^{2})\rightarrow L^{0}_{+}(\mathcal {F})$ by $\|T\|:=\bigwedge\{\xi\in L^{0}_{+}(\mathcal{F})~|~\|Tx\|_{2}\leq \xi\cdot\|x\|_{1},~\forall x\in S^{1}\}$ for any $T\in B(S^{1},S^{2})$, then clearly $(B(S^{1},S^{2}),\|\cdot\|)$ is an $RN$ module.
\end{definition}

The following proposition shows that an a.s. bounded random linear operator on an $RN$ module $S$ is exactly a continuous module homomorphism on $S$.

\begin{proposition}\rm(see \cite{Guohomomorphism})
Let $(S^{1},\|\cdot\|_{1})$ and $(S^{2},\|\cdot\|_{2})$ be two $RN$ modules over $K$ with base $(\Omega,{\mathcal F},P)$. Then we have the following statements:

(1) $T\in B(S^{1},S^{2})$ if and only if $T$ is a continuous module homomorphism;

(2) If $T\in B(S^{1},S^{2})$, then $\|T\|=\bigvee\{\|Tx\|_{2}:x \in S^{1} ~\textrm{and}~\|x\|_{1}\leq 1\},$ where 1 denotes the unit element in $L^{0}({\mathcal F},R)$.
\end{proposition}

\begin{definition}\rm(see \cite{Thang})
Let $S$ be an $RN$ module. A mapping $f:[a,b]\rightarrow S$ is said to be $L^{0}$-Lipschitz on a finite real closed interval $[a,b]$ if there exists a $\xi\in L^{0}_{+}(\mathcal{F})$ such that
$\|f(t_{1})-f(t_{2})\|\leq \xi|t_{1}-t_{2}|,~~~~~~\forall t_{1},t_{2}\in[a,b].$
\end{definition}

\begin{remark}
It is easy to check that a mapping $f:[a,b]\rightarrow S$ is $L^{0}$-Lipschitz on $[a,b]$ if and only if the mapping $f$ satisfies the assumption
$\bigvee\{\|\frac{f(t_{1})-f(t_{2})}{t_{1}-t_{2}}\|~|~t_{1},~t_{2}\in[a,b]~\textrm{and}~t_{1} \neq t_{2}\}\in L^{0}_{+}({\mathcal F}).$
\end{remark}

\begin{proposition}(see \cite{Guozhang}) (The fundamental theorem of calculus) Let $S$ be a complete $RN$ module and $f:[a,b]\rightarrow S$ a continuously differentiable function. If $f$ is $L^{0}$-Lipschitz on $[a,b]$, then $f^{'}$ is Riemann integrable and
$f(b)-f(a)=\int_{a}^{b}f^{'}(t)dt.$
\end{proposition}

\section{A characterization for a $C_{0}$--semigroup of continuous module homomorphisms to be almost surely bounded}\label{section3}

The purpose of this section is to give a characterization for a $C_{0}$--semigroup of continuous module homomorphisms to be almost surely bounded, which reflects the nature of a $C_{0}$--semigroup of continuous module homomorphisms. For the reader's convenience, let us first recall some known results as follows.

\begin{definition}\rm(see \cite{Zhang})
Let $(S,\|\cdot\|)$ be an $RN$ module, $B(S)$ the set of continuous module homomorphisms from $S$ to $S$. Then a family $\{T(t):t\geq0\}\subset B(S)$ is called a semigroup of continuous module homomorphisms if $T(0)=I~~\textrm{and}~~T(s)T(t)=T(s+t)$ for all $s,t\geq0$, where $I$ denotes the identity operator on $S$. Moreover, the semigroup of continuous module homomorphisms $\{T(t):t\geq0\}$ is said to be strongly continuous if $\lim_{t\downarrow0}T(t)x=x$ for any $x\in S.$ Besides, a strongly continuous semigroup of continuous module homomorphisms on $S$ is also called a $C_{0}$--semigroup.
\end{definition}

\begin{definition}\rm(see \cite{ZhangLiu})
Let $\{T(t):t\geq0\}$ be a semigroup of continuous module homomorphisms on an $RN$ module $S$. Define $D(A)=\{x\in S:\lim_{t\downarrow0}\frac{T(t)x-x}{t}~~exists\}$ and $Ax=\lim_{t\downarrow0}\frac{T(t)x-x}{t}$ for any $x\in D(A)$. Then the mapping $A:D(A)\rightarrow S$ is called the infinitesimal generator of $\{T(t):t\geq0\}$, also denoted by $(A,D(A))$ in this paper.
\end{definition}

In the sequel of this paper, we always assume that $(S,\|\cdot\|)$ is a complete $RN$ module over $K$ with base $(\Omega,{\mathcal F},P)$.
It is well known that a continuous function from a finite real closed interval to a Banach space is bounded. But unlike the classical case, Example 2.3 in \cite{Guozhang} shows that, for a $C_{0}$--semigroup $\{T(t):t\geq0\}$, and for any $x\in S$, $\{\|T(t)x\|:t\in [a,b]\}$ may not be a.s. bounded, so it is necessary to define the notions of a.s. bounded and a.s.u. bounded $C_{0}$--semigroups.

\begin{definition}(see \cite{Zhang})
A $C_{0}$--semigroup $\{T(t):t\geq0\}$ is said to be almost surely uniformly bounded (briefly, a.s.u. bounded) if $\bigvee_{t\geq0}\ {\|T(t)\|}$ belongs to $L^{0}_{+}({\mathcal F})$, i.e., there exists a $\xi\in L^{0}_{+}({\mathcal F})$ such that $\|T(t)x\|\leq\xi\|x\|,~~\forall t\geq0~~and~~x\in S.$
\end{definition}

Set $A_{t}=\frac{T(t)-T(0)}{t}~~(\forall t>0),$ and $A$ denotes the infinitesimal generator of $\{T(t):t\geq0\}$, i.e., $Ax=\lim_{t\downarrow0}A_{t}x$ for any $x\in D(A)$. It is known from \cite{Zhang} that in general $A$ is a module homomorphism and $D(A)$ is dense in $S$. The following Proposition 3.4 shows that for any $x\in D(A)$, $\{A_{t}x,~t>0\}$ is a.s. bounded. It should be pointed out that this special property of a.s.u. bounded
$C_{0}$--semigroups plays a crucial role in the proof of Lemma 3.7.

\begin{proposition}(see \cite{ZhangLiu})
Let $\{T(t):t\geq0\}$ be an a.s.u. bounded $C_{0}$--semigroup with the infinitesimal generator $(A, D(A))$. Then for any $x\in D(A)$, $\bigvee_{t>0}\ {\|A_{t}x\|}$ belongs to $L^{0}_{+}({\mathcal F})$.
\end{proposition}

\begin{theorem}(see \cite{ZhangLiu})
 Let $(S,\|\cdot\|)$ be a complete $RN$ module over $K$ with base $(\Omega,{\mathcal F},P)$ and $\{T(t):t\geq0\}$ an a.s.u. bounded $C_{0}$--semigroup of continuous module homomorphisms on $S$. Then the infinitesimal generator $A$ of $\{T(t):t\geq0\}$ is densely defined on $S$ and satisfies
$\frac{dT(t)x}{dt}=AT(t)x=T(t)Ax,~\forall x\in D(A)$
and
$T(r)x-x=\int_{0}^{r}T(s)Axds=\int_{0}^{r}AT(s)xds,~\forall x\in D(A)~and~r>0.$
\end{theorem}

In 2019, in order to generalize the above Theorem 3.5, Thang et.al proposed Definition 3.6 below.

\begin{definition}\rm (see \cite{Thang})
A $C_{0}$--semigroup $\{T(t):t\geq0\}$ is said to be a.s. bounded if there exists a finite real number $L>0$ such that $\bigvee_{t\in[0,L]}\ {\|T(t)\|}\in L^{0}_{+}({\mathcal F}),$ i.e., there exists a $\xi_{L}\in L^{0}_{+}({\mathcal F})$ such that $\|T(t)x\|\leq \xi_{L}\|x\|,~~~~\forall t\in[0,L]~~and~~x\in S.$
\end{definition}

In \cite{Thang}, the most crucial step is to generalize the above Proposition 3.4 from an a.s.u. bounded
 $C_{0}$--semigroup to an a.s. bounded $C_{0}$--semigroup. Unfortunately, the course of this generalization is not strict. In the following, we will give a new and strict proof of this generalization.

\begin{lemma}
Let $\{T(t):t\geq0\}$ be an a.s. bounded $C_{0}$--semigroup on $S$ with the infinitesimal generator $(A,D(A))$. Given $x\in D(A)$, define a function $f: [0, +\infty)\rightarrow S$ by
$f(t)=T(t)x$. Then $f$ is $L^{0}-$Lipschitz on any finite real closed interval $[0,r]$ for any given $x \in D(A)$.
\end{lemma}

\begin{proposition}(see \cite{Thang})
Let $\{T(t):t\geq0\}$ be an a.s. bounded $C_{0}$--semigroup on $S$. Then there exist $M,\tau \in L^{0}_{+}({\mathcal F})$ and $M\geq1$ such that $\|T(t)\|\leq Me^{\tau t},~\forall t\in[0,\infty).$
\end{proposition} We can now prove Lemma 3.7 below.\begin{proof}[Proof of Lemma 3.7 ]\label{lemma3.1proof1}
Since $\{T(t):t\geq0\}$ is a.s. bounded, it follows from Proposition 3.8 that there exist $M,\tau\in L^{0}_{+}({\mathcal F})$ and $M\geq1$ such that $\|T(t)\|\leq Me^{\tau t}$. Further, it is easy to see that $f$ is continuously differentiable and $f^{'}(t)=T(t)Ax=AT(t)x$ for any $t\geq 0$.
Set $\widetilde{T}(t)=e^{-\tau t}T(t)$, then it is easy to check that $\{\widetilde{T}(t):t\geq0\}$ is a
$C_{0}$--semigroup with the infinitesimal generator $\widetilde{A}$, where $\widetilde{A}=A-\tau I$. Moreover, $\|\widetilde{T}(t)\|=\|e^{-\tau t}T(t)\|=e^{-\tau t}\|T(t)\|\leq M,$ i.e.,
\begin{align}
\bigvee_{t\geq0}\|\widetilde{T}(t)\|\in L^{0}_{+}({\mathcal F}),
\end{align}
then it follows from Proposition 3.4 that
\begin{align}
\bigvee_{t>0}\|\frac{\widetilde{T}(t)-I}{t}x\|\in L^{0}_{+}({\mathcal F}).
\end{align}
Thus we have
\begin{align}
&\bigvee\{\|\frac{f(t_{1})-f(t_{2})}{t_{1}-t_{2}}\|~|~t_{1},~t_{2}\in[0,r]~\textrm{and}~t_{1}> t_{2}\}\nonumber\\
&=\bigvee\{\|\frac{T(t_{1})x-T(t_{2})x}{t_{1}-t_{2}}\|~|~t_{1},~t_{2}\in[0,r]~\textrm{and}~t_{1}>t_{2}\}\nonumber\\
&=\bigvee\{\|\frac{e^{\tau t_{1}}\widetilde{T}(t_{1})x-e^{\tau t_{2}}\widetilde{T}(t_{2})x}{t_{1}-t_{2}}\|~|~t_{1},~t_{2}\in[0,r]~\textrm{and}~t_{1}> t_{2}\}\nonumber\\
&=\bigvee\{\|e^{\tau t_{2}}\widetilde{T}(t_{2})\frac{e^{\tau(t_{1}-t_{2})}\widetilde{T}(t_{1}-t_{2})x-x}{t_{1}-t_{2}}\|~|~t_{1},~t_{2}\in[0,r]~\textrm{and}~t_{1}> t_{2}\}\nonumber\\
&\leq e^{\tau r}\cdot\bigvee_{t\geq0}\|\widetilde{T}(t)\|\cdot\bigvee\{\|\frac{e^{\tau(t_{1}-t_{2})}\widetilde{T}(t_{1}-t_{2})x-x}{t_{1}-t_{2}}\|~|~t_{1},~t_{2}\in[0,r]~\textrm{and}~t_{1}> t_{2}\}\nonumber\\
&\leq e^{\tau r}\cdot\bigvee_{t\geq0}\|\widetilde{T}(t)\|\cdot\bigvee\{\|\frac{e^{\tau t}\widetilde{T}(t)-I}{t}x\|~|~t\in(0,r]\}\nonumber
\vspace{10mm}
\end{align}
\begin{align}
&=e^{\tau r}\cdot\bigvee_{t\geq0}\|\widetilde{T}(t)\|\cdot\bigvee\{\|\frac{e^{\tau t}\widetilde{T}(t)-e^{\tau t}I+e^{\tau t}I-I}{t}x\|~|~t\in(0,r]\}\nonumber\\
&\leq e^{\tau r}\cdot\bigvee_{t\geq0}\|\widetilde{T}(t)\|\cdot\{e^{\tau r}\bigvee_{0<t\leq r}\|\frac{\widetilde{T}(t)-I}{t}x\|+\bigvee\{|\frac{e^{\tau t}-1}{\tau t}|~\tau~\|x\|~|~t\in(0,r]\}\}\nonumber\\
&\leq e^{\tau r}\cdot\bigvee_{t\geq0}\|\widetilde{T}(t)\|\cdot\{e^{\tau r}\bigvee_{t>0}\|\frac{\widetilde{T}(t)-I}{t}x\|+\bigvee\{\tau e^{\tau t}\|x\|~|~t\in(0,r]\}\}\nonumber\\
&\leq e^{2\tau r}\cdot\bigvee_{t\geq0}\|\widetilde{T}(t)\|\cdot\{\bigvee_{t>0}\|\frac{\widetilde{T}(t)-I}{t}x\|+\tau\|x\|\}\nonumber\\
&\in L^{0}_{+}({\mathcal F})\nonumber~~~~(according~ to~ (3.1)~and~(3.2)),
\end{align}
i.e., $f(t)$ is $L^{0}-$Lipschitz on any finite real closed interval $[0,r]$,
which completes the proof of Lemma 3.7.
\end{proof}

Based on Theorem 3.5 and Lemma 3.7, we can obtain the following Theorem 3.9.

\begin{theorem}
 Let $\{T(t):t\geq0\}$ be an a.s. bounded $C_{0}$--semigroup on $S$ and $(A, D(A))$ its infinitesimal generator. Then $A$ is a densely defined module homomorphism on $S$ and satisfies



(a) $$\frac{dT(t)x}{dt}=AT(t)x=T(t)Ax$$
for any $x\in D(A)$.

(b) $$T(r)x-x=\int_{0}^{r}T(s)Axds=\int_{0}^{r}AT(s)xds$$
for any $x\in D(A)$ and $r>0$.

\end{theorem}

In order to give a characterization for a $C_{0}$--semigroup of continuous module homomorphisms to be almost surely bounded,
let us first recall the resonance theorem under the $(\varepsilon,\lambda)-$topology as follows.
\begin{proposition}(see \cite{Guohomomorphism})
Let $(S_{1},\|\cdot\|_{1})$ and $(S_{2},\|\cdot\|_{2})$ be two $RN$ modules over $K$ with base $(\Omega,{\mathcal F},P)$ such that $S_{1}$ is complete and $\{T_{\alpha}:\alpha\in\Lambda\}$ a family of continuous module homomorphisms from $(S_{1},\|\cdot\|_{1})$ to $(S_{2},\|\cdot\|_{2})$. Then $\{T_{\alpha}:\alpha\in\Lambda\}$ is a.s. bounded in $B(S_{1},S_{2})$ if and only if $\{T_{\alpha}(x):\alpha\in\Lambda\}$ is a.s. bounded in $S_{2}$ for each $x\in S_{1}$.
\end{proposition}

\begin{theorem}
Let $\{T(t):t\geq0\}$ be a $C_{0}$--semigroup on $S$ with the infinitesimal generator $(A,D(A))$ and $f$ be the same as in Lemma 3.7.
Then $\{T(t):t\geq0\}$ is a.s. bounded if and only if $f$ is $L^{0}-$Lipschitz on any finite real closed interval $[0,r]$ for any given $x \in D(A)$.
\end{theorem}

\begin{proof}
$``\Rightarrow"$ It is obvious from Lemma 3.7.\\
$``\Leftarrow"$ Let $$\xi=\bigvee\{\|\frac{T(t_{1})x-T(t_{2})x}{t_{1}-t_{2}}\|~|~t_{1},~t_{2}\in[0,r]~\textrm{and}~t_{1}>t_{2}\}$$
for any $x\in D(A)$ and any finite interval [0,r], then $\xi\in L^{0}_{+}({\mathcal F})$ and
$$\|T(t_{1})x-T(t_{2})x\|\leq \xi \cdot |t_{1}-t_{2}|$$
for any $t_{1}, t_{2}\in [0,r]$, i.e.,
$\|T(t)x-x\|\leq \xi \cdot t$
for any $t\in [0,r].$ Therefore,
$$\|T(t)x\|\leq\|T(t)x-x\|+\|x\|
\leq \xi \cdot t+\|x\|$$
for any $t\in [0,r],$ i.e., the set $\{T(t)x,t\in [0,r]\}$ is a.s. bounded in $S$ for any $x\in D(A)$. Since $D(A)$ is dense in $S$, it follows that for any $y\in S$, there exists a sequence $\{x_{n}, n\in N\}\subset D(A)$ such that $x_{n}\rightarrow y$ as $n\rightarrow \infty$. Observing that
$$\|T(t)x_{n}\|\leq \xi \cdot t+\|x_{n}\|$$for any $t\in [0,r],$
letting $n\rightarrow \infty$ in the above inequality, we have
$\|T(t)y\|\leq \xi \cdot t+\|y\|$ for any $t\in [0,r],$ which shows that the set $\{T(t)y,t\in [0,r]\}$ is a.s. bounded in $S$ for any $y\in S$. Since $S$ is complete, it follows from Proposition 3.10 that
$\{T(t),t\in [0,r]\}$ is a.s. bounded in $B(S)$, i.e.,
$\bigvee_{t\in[0,r]}\|T(t)\|\in L^{0}_{+}({\mathcal F}),$
which completes the proof of Theorem 3.11.
\end{proof}

The following example shows that it is necessary to require the $L^{0}-$Lipschitz assumption in Proposition 2.7, one can also refer to \cite{Thang}.

\begin{example}
Let $\Omega=[0,1],~\mathcal
{F}=\mathfrak{B}[0,1],~P=m$, where $\mathfrak{B}[0,1]$ denotes the
Borel $\sigma-$algebra on $[0,1]$ and $m$ the Lesbegue measure.
Define a mapping $f^{0}:[0,1]\rightarrow \mathcal {L}^{0}(\mathcal
{F},R)$ by $f^{0}(t)(\omega)= I_{(t,1]}(\omega)$ for any $t\in[0,1]$ and $\omega\in \Omega$.
For any fixed $t\in[0,1]$, let $f(t)$ denote the
equivalence class determined by $f^{0}(t)$. Then $f$ is differentiable
on $[0,1]$ and $f^{'}(t)=0,~~~\forall t\in [0,1].$
\end{example}

\begin{proof} Since $|f^{0}(t)(\omega)-f^{0}(t_{0})(\omega)|=$\[
|I_{(t,1]}(\omega)-
I_{(t_{0},1]}(\omega)|=I_{(t_{0}\wedge t,~t_{0}\vee t]}(\omega)= \left\{
\begin{array}
    {l@{\qquad}l}
     1  &~ t_{0}\wedge t<\omega\leq t_{0}\vee t,\\
    0 &  ~otherwise,

\end{array}
\right.
\]for any $t,~t_{0}\in [0,1]$.
Thus for any $\varepsilon >0$ and $t\neq t_{0}$,
\begin{align}
P(|\frac{f^{0}(t)(\omega)-f^{0}(t_{0})(\omega)}{t-t_{0}}|>\varepsilon)&\leq P(|f^{0}(t)(\omega)-f^{0}(t_{0})(\omega)|=1)\nonumber\\
&=  P(\omega~|~t_{0}\wedge t<\omega\leq t_{0}\vee t)\nonumber\\
&\leq |t-t_{0}|,\nonumber
\end{align}
which shows that $(f^{0})'(t_{0})=0$, i.e., $f'(t_{0})=0,$ thus $f'(t)=0,~\forall t\in [0,1]$. Further we have $\int_{0}^{1}f^{'}(t)dt=0$, but $f(1)-f(0)=-1.$
\end{proof}

\begin{remark}
Theorem 3.11 and Example 3.12 show that it is necessary to require the almost sure boundedness for such a
$C_{0}$--semigroup on $S$ in Theorem 3.9. In fact, if the $C_{0}$--semigroup on $S$ in Theorem 3.9 is not a.s. bounded, then, according to Theorem 3.11, $f(t)$ is not $L^{0}-$Lipschitz on some finite interval $[a,b]$. Thus, according to Example 3.12, the fundamental theorem of calculus may not hold, that is to say, the part (b) of Theorem 3.9 may not hold, which implies that it is necessary to require the almost sure boundedness for such a $C_{0}$--semigroup in Theorem 3.9.
\end{remark}

\section{The Hille-Yosida generation theorems on complete random normed modules}\label{section4}
The central result of this section Theorem 4.7, which is
called the Hille-Yosida generation theorem for an a.s. bounded $C_{0}$--semigroup,
is devoted to establishing three equivalent conditions for a module homomorphism to generate an a.s. bounded $C_{0}$--semigroup. For the sake of clearness, the proof of Theorem 4.7 is divided into three lemmas, i.e., Lemmas 4.8, 4.9 and 4.11. In particular, in the proof of Lemma 4.8, we are forced to deal with a difficult point with respect to the $L^{0}$-norm transformation technique on a complete $RN$ module. Before presenting them, let us give some preliminaries for the reader's convenience.

Let $A : D(A)\subset S\rightarrow S$ be a module homomorphism and
$$\rho(A)=\{\xi\in L^{0}({\mathcal F},K) : \xi I-A~\textrm{is~a~bijective~and~}{(\xi I-A)}^{-1}\in B(S)\},$$
then $\rho(A)$ is called the resolvent set of $A$. Further, if $\xi \in \rho(A)$, then $R(\xi,A)=(\xi I-A)^{-1}$ is called the resolvent of A.
It is easy to check that $R(\xi,A)$ is exactly the following mapping defined by
$R(\xi,A)x=\int_{0}^{+\infty}e^{-\xi s}T(s)xds$
for any $x\in S$ and $\xi\in L^{0}({\mathcal F},C)$ with $Re(\xi)\in L_{++}^{0}(\mathcal F)$, one can also see \cite{Thang} for details.

\begin{proposition}\rm(see \cite{Thang})
Let $A:D(A)\subset S\rightarrow S$ be a module homomorphism. Then $(A,D(A))$ is the infinitesimal generator of a contraction $C_{0}$--semigroup if and only if

(i) $A$ is closed and $D(A)$ is dense in $S$.

(ii) The resolvent set $\rho(A)$ contains the set $\{\xi\in L^{0}({\mathcal F},C)~|~Re\xi\in L_{++}^{0}(\mathcal F)\}$ and for such a $\xi$,
$\|R(\xi,A)\|\leq\frac{1}{Re\xi}.$
\end{proposition}

\begin{proposition}(see \cite{Thang})
Let $\{T(t):t\geq0\}$ and $\{S(t):t\geq0\}$ be two a.s. bounded $C_{0}$--semigroup with infinitesimal generators $(A,D(A))$ and $(B,D(B))$ respectively on $S$.
If $A=B$, then $T(t)=S(t)$ for any $t\geq0$.
\end{proposition}

\begin{remark}\rm
According to Proposition 4.2 and Theorem 3.9, one can obtain that the infinitesimal generator of an a.s. bounded $C_{0}$--semigroup is a  module homomorphism that determines the semigroup uniquely. Consequently, if $A:D(A)\subset S\rightarrow S$ is the infinitesimal generator of an a.s. bounded $C_{0}$--semigroup $\{T(t):t\geq0\}$, then we can also say the module homomorphism $(A, D(A))$ generates the $C_{0}$--semigroup $\{T(t):t\geq0\}$.
\end{remark}

Based on Proposition 4.1 and Remark 4.3, we can state the Hille-Yosida generation theorem for a contraction
$C_{0}$--semigroup as follows.

\begin{theorem}
For a module homomorphism $(A,D(A))$ on $S$, the following assertions are all equivalent.

(a) $(A,D(A))$ generates a contraction $C_{0}$--semigroup.

(b) $(A,D(A))$ is closed, densely defined, and for each $\xi\in L_{++}^{0}(\mathcal F)$ one has $\xi\in\rho(A)$ and $\|\xi R(\xi,A)\|\leq1.$

(c) $(A,D(A))$ is closed, densely defined, and for each $\xi\in L^{0}(\mathcal F,C)$ with $Re\xi\in L_{++}^{0}(\mathcal F)$ one has $\xi\in\rho(A)$ and
$\|R(\xi,A)\|\leq\frac{1}{Re\xi}.$
\end{theorem}

If a $C_{0}$--semigroup $\{T(t):t\geq0\}$ with infinitesimal generator $A$ satisfies, for some $\tau\in L^{0}(\mathcal F,R)$,
$\|T(t)\|\leq e^{\tau t}$ for any $t\geq0$, then we can apply the above characterization to the rescaled contraction semigroup given by
$S(t)=e^{-\tau t}T(t)$
for any $t\geq0$. Since it is clear that the infinitesimal generator of $\{S(t):t\geq0\}$ is $A-\tau I$, Theorem 4.4 takes the following form.

\begin{corollary}\rm
Suppose that $\tau\in L^{0}(\mathcal F,R)$. For a module homomorphism $(A,D(A))$ on $S$, the following assertions are equivalent.

(a) $(A,D(A))$ generates an a.s. bounded $C_{0}$--semigroup $\{T(t):t\geq0\}$ satisfying
$\|T(t)\|\leq e^{\tau t}$
for any $t\geq0$.

(b) $(A,D(A))$ is closed, densely defined, and for each $\xi\in L^{0}(\mathcal F,R)$ with $(\xi-\tau)\in L_{++}^{0}(\mathcal F)$ one has $\xi\in \rho(A)$ and
$\|(\xi-\tau)R(\xi,A)\|\leq1.$

(c) $(A,D(A))$ is closed, densely defined, and for each $\xi\in L^{0}(\mathcal F,C)$ with $(Re\xi-\tau)\in L_{++}^{0}(\mathcal F)$ one has $\xi\in \rho(A)$ and
$\|R(\xi,A)\|\leq\frac{1}{Re\xi-\tau}.$
\end{corollary}

\begin{remark}\label{rem3.1}
It should be pointed out that it is necessary to require the almost sure boundedness in Corollary 4.5, which is quite different from the classical case.
\end{remark}

The main result of this section is the following Theorem 4.7.

\begin{theorem}Let $(A,D(A))$ be a module homomorphism on $S$. Suppose that $\tau,M\in L^{0}(\mathcal F,R)$ and $M\geq1$. Then the following assertions are equivalent.

(a) $(A,D(A))$ generates an a.s. bounded $C_{0}$--semigroup $\{T(t):t\geq0\}$ satisfying
$\|T(t)\|\leq Me^{\tau t}$
for any $t\geq0$.

(b) $(A,D(A))$ is closed, densely defined, and for each $\xi\in L^{0}(\mathcal F,R)$ with $(\xi-\tau)\in L_{++}^{0}(\mathcal F)$ one has $\xi\in \rho(A)$ and
$\|[(\xi-\tau)R(\xi,A)]^{n}\|\leq M$
for any $n\in N$.

(c) $(A,D(A))$ is closed, densely defined, and for each $\xi\in L^{0}(\mathcal F,C)$ with $(Re\xi-\tau)\in L_{++}^{0}(\mathcal F)$ one has $\xi\in \rho(A)$ and
$$\|R(\xi,A)^{n}\|\leq\frac{M}{(Re\xi-\tau)^{n}}$$
for any $n\in N$.
\end{theorem}

The proof of Theorem 4.7 needs the following three lemmas.

\begin{lemma}
Let $A:D(A)\rightarrow S$ be a module homomorphism for which $L^{0}_{++}({\mathcal F})\subset\rho(A)$. If there exists an $M\in L^{0}_{+}({\mathcal F})$ with $M \geq 1$ such that
\begin{align}
\|\xi^{n}R(\xi,A)^{n}\|\leq M
\end{align}
for any $n\in N$ and $\xi\in L^{0}_{++}({\mathcal F})$, then there exists an $L^{0}$-norm $|\cdot|$ on $S$ which is equivalent to the original $L^{0}$-norm $\|\cdot\|$ on $S$ and satisfies
\begin{align}
\|x\|\leq|x|\leq M\|x\|
\end{align}
for any $x\in S$ and
\begin{align}
|\xi R(\xi,A)x|\leq|x|
\end{align}
for any $x\in S$ and $\xi\in L^{0}_{++}({\mathcal F})$.
\end{lemma}

\begin{proof} Let $\eta\in L^{0}_{++}({\mathcal F})$, define a mapping $\|\cdot\|_{\eta}:S\rightarrow L^{0}_{+}({\mathcal F})$ by
\begin{align}
\|x\|_{\eta}=\bigvee_{n\in N \cup \{0\}}\|\eta^{n}R(\eta,A)^{n}x\|,
\end{align}
then it is clear that $\|\cdot\|_{\eta}$ is an $L^{0}$-norm on $S$ and further
\begin{align}
\|x\|\leq\|x\|_{\eta}\leq M\|x\|
\end{align}
and
\begin{align}
\|\eta R(\eta,A)\|_{\eta}\leq1.
\end{align}

We claim that
\begin{align}
\|\xi R(\xi,A)\|_{\eta}\leq1
\end{align}
for any $\xi\in L^{0}_{++}({\mathcal F})$ and $\xi\leq\eta$.

In fact, if $y=R(\xi,A)x$, then $y=R(\eta,A)(x+(\eta-\xi)y)$ by the resolvent equation, one can obtain by (4.6) that
\begin{align}
\|y\|_{\eta}&=\|R(\eta,A)(x+(\eta-\xi)y)\|_{\eta}\nonumber\\
&\leq \|R(\eta,A)x\|_{\eta}+\|(\eta-\xi)R(\eta,A)y\|_{\eta}\nonumber\\
&\leq \frac{1}{\eta}\|\eta R(\eta,A)x\|_{\eta}+\frac{\eta-\xi}{\eta}\|\eta R(\eta,A)y\|_{\eta}\nonumber\\
&\leq \frac{1}{\eta}\|x\|_{\eta}+(1-\frac{\xi}{\eta})\|y\|_{\eta}\nonumber
\end{align}
whence $\xi\|y\|_{\eta}\leq\|x\|_{\eta}$ as claimed. From (4.5) and (4.6) it follows that
\begin{align}
\|\xi^{n}R(\xi,A)^{n}x\|\leq\|\xi^{n}R(\xi,A)^{n}x\|_{\eta}\leq\|x\|_{\eta}
\end{align}
for any $\xi\in L^{0}_{++}({\mathcal F})$ and $\xi\leq\eta$.
Thus $$\bigvee_{n\in N \cup \{0\}}\|\xi^{n}R(\xi,A)^{n}x\|\leq\|x\|_{\eta},$$
i.e.,
\begin{align}
\|x\|_{\xi}\leq\|x\|_{\eta}
\end{align}
for any $\xi\in L^{0}_{++}({\mathcal F})$ and $\xi\leq\eta$.

Next, define a mapping $|\cdot|:S\rightarrow L^{0}_{+}({\mathcal F})$ by
\begin{align}
|x|=\bigvee\{\|x\|_{\xi}~|~\xi\in L^{0}_{++}({\mathcal F})\}
\end{align}
for any $x\in S$, then clearly $|\cdot|$ is well defined and it is easy to check that $|\cdot|$ is an $L^{0}$-norm.

It follows from (4.5) that (4.2) holds. As for (4.3), it follows from (4.7) and (4.9) for any $\eta \in L^0_{++}(\mathcal{F})$ that $$\|\xi R(\xi,A)x\|_{\eta} \leq \|\xi R(\xi,A)x\|_{\eta \vee \xi} \leq \|x\|_{\eta \vee \xi} \leq |x|,$$ and hence $$|\xi R(\xi,A)x| \leq |x|.$$

This completes the proof.

%
\end{proof}

\begin{lemma}
Let $A:D(A)\rightarrow S$ be a module homomorphism. Then $(A,D(A))$ is the infinitesimal generator of a
$C_{0}$--semigroup $\{T(t):t\geq0\}$ satisfying $\|T(t)\|\leq M$ $(M\in L^{0}_{++}({\mathcal F})~and~M\geq1)$ if and only if

(i) $A$ is closed and $D(A)$ is dense in $S$.

(ii) The resolvent set $\rho(A)$ of $A$ contains $L^{0}_{++}({\mathcal F})$ and
\begin{align}
\|\xi^{n}R(\xi,A)^{n}\|\leq M
\end{align}
for any $\xi\in L^{0}_{++}({\mathcal F})$ and $n\in N$.
\end{lemma}

\begin{proof} Let $\{T(t):t\geq0\}$ be a $C_{0}$--semigroup on $S$ and $(A,D(A))$ its infinitesimal generator. If the $L^{0}$-norm on $S$ is changed to an equivalent $L^{0}$-norm, then clearly $\{T(t):t\geq0\}$ remains a $C_{0}$--semigroup on $S$ with the equivalent $L^{0}$-norm. The infinitesimal generator $(A,D(A))$ does not change, and the fact that $A$ is closed and densely defined does not change, either, when we pass to an equivalent $L^{0}$-norm on $S$.
All these are topological properties which are independent of a particular equivalent $L^{0}$-norm with which $S$ is endowed.

Now, let $(A,D(A))$ be the infinitesimal generator of a $C_{0}$--semigroup satisfying $\|T(t)\|\leq M$. Define a mapping $|\cdot|:S\rightarrow L^{0}_{+}({\mathcal F})$ by
\begin{align}
|x|=\bigvee_{t\geq0}\|T(t)x\|,
\end{align}
then $|\cdot|$ is well defined and clearly $|\cdot|$ is an $L^{0}$-norm on $S$.
Moreover
\begin{align}
\|x\|\leq|x|\leq M\|x\|
\end{align}
and therefore the $L^{0}$-norm $|\cdot|$ is equivalent to the original $L^{0}$-norm $\|\cdot\|$ on $S$. Since
\begin{align}
|T(t)x|&=\bigvee_{s\geq0}\|T(s)T(t)x\|\nonumber\\
&=\bigvee_{s\geq0}\|T(s+t)x\|\nonumber\\
&\leq \bigvee_{s\geq0}\|T(s)x\|\nonumber\\
&=|x|,
\end{align}it follows that $\{T(t):t\geq0\}$ is a $C_{0}$--semigroup of contractions on $S$ endowed with the $L^{0}$-norm $|\cdot|$. Thus, according to Theorem 4.4 and the remarks at the beginning of the proof, it follows that $A$ is closed and $D(A)$ is dense in $S$ and that
\begin{align}|R(\xi,A)|\leq\frac{1}{\xi}
\end{align}
for any $\xi\in L^{0}_{++}({\mathcal F})$. Consequently, combining Inequalities (4.13) and (4.15), we have
\begin{align}
\|R(\xi,A)^{n}x\|\leq|R(\xi,A)^{n}x|
\leq\frac{1}{\xi^{n}}|x|
\leq\frac{M}{\xi^{n}}\|x\|,\nonumber
\end{align}
which implies that the conditions (i) and (ii) are necessary.

Let the conditions (i) and (ii) be satisfied. Then, according to Lemma 4.8, there exists an $L^{0}$-norm $|\cdot|$ on $S$ satisfying Inequalities (4.2) and (4.3). Now, considering $S$ with the $L^{0}$-norm $|\cdot|$, it is clear that $A$ is a closed densely defined module homomorphism with $L^{0}_{++}({\mathcal F})\subset\rho(A)$ and $|R(\xi,A)|\leq\frac{1}{\xi}$ for any $\xi\in L^{0}_{++}({\mathcal F})$. Thus, according to Theorem 4.4, $A$ is also the infinitesimal generator of a $C_{0}$--semigroup of contractions on $S$ endowed with the norm $|\cdot|$. Returning to the origin $L^{0}$-norm $\|\cdot\|$, $A$ is again the infinitesimal generator of $\{T(t):t\geq0\}$ and further
\begin{align}
\|T(t)x\|\leq|T(t)x|
\leq|x|
\leq M\|x\|\nonumber
\end{align}
for any $x\in S$ so $\|T(t)\|\leq M$ as required. Therefore, the conditions (i) and (ii) are also sufficient.

This completes the proof of Lemma 4.9.
\end{proof}

\begin{remark}
Let $(A,D(A))$ be the infinitesimal generator of an a.s.u. bounded $C_{0}$--semigroup $\{T(t):t\geq0\}$, i.e., $\bigvee_{t\geq0}\|T(t)\|\in L^{0}_{+}({\mathcal F})$. Denote $M=\bigvee_{t\geq0}\|T(t)\|$, then $M\in L^{0}_{++}({\mathcal F})$ and $M\geq1$. Then the resolvent set of $A$ contains the set $\{\xi\in L^{0}({\mathcal F},C)~|~Re\xi\in L^{0}_{++}({\mathcal F})\}$ and for such $\xi$,
$\|R(\xi,A)\|\leq\frac{M}{Re\xi}.$
\end{remark}

In fact, for any $\xi\in L^{0}({\mathcal F},C)$ and $Re\xi\in L^{0}_{++}({\mathcal F})$,
\begin{align}
\|R(\xi,A)x\|&=\|\int_{0}^{+\infty}e^{-\xi t}T(t)xdt\|\nonumber\\
&\leq\int_{0}^{+\infty}|e^{-\xi t}|~\|T(t)x\|dt\nonumber\\
&\leq\int_{0}^{+\infty}e^{-Re\xi\cdot t}\cdot M\cdot\|x\|dt\nonumber\\
&=\frac{M}{Re\xi}\|x\|,\nonumber
\end{align}
thus the module homomorphism $R(\xi,A)x=\int_{0}^{+\infty}e^{-\lambda t}T(t)xdt$ is well-defined for $\xi\in L^{0}({\mathcal F},C)$ satisfying $Re\xi\in L^{0}_{++}({\mathcal F})$.
Further, according to Inequations (4.13) and (4.15),
$$\|R(\xi,A)^{n}x\|\leq|R(\xi,A)^{n}x|\leq\frac{1}{(Re\xi)^{n}}|x|\leq M\cdot\frac{1}{(Re\xi)^{n}}.$$
Consequently, based on the above remarks and Lemma 4.8, one can immediately obtain the following Lemma 4.11, which generalizes Proposition 4.1.

\begin{lemma}
Let $A:D(A)\rightarrow S$ be a module homomorphism. Then $A$ is the infinitesimal generator of a $C_{0}$--semigroup $\{T(t):t\geq0\}$, satisfying $\|T(t)\|\leq M$ $(M\in L^{0}_{++}({\mathcal F})~and~M\geq1)$, if and only if

(i) $A$ is closed and $D(A)$ is dense in $S$.

(ii) The resolvent set $\rho(A)$ of $A$ contains the set $\{\xi\in L^{0}({\mathcal F},C)~|~Re\xi\in L^{0}_{++}({\mathcal F})\}$ and for such $\xi$,
$\|R(\xi,A)^{n}\|\leq\frac{M}{(Re\xi)^{n}}$
for any $n\in N$.
\end{lemma}

Based on Lemmas 4.9 and 4.11, we can now prove Theorem 4.7 as follows.

\begin{proof}[Proof of Theorem 4.7] $(a)\Leftrightarrow (c).$  Since $\{T(t):t\geq0\}$ is a.s. bounded, it follows from Proposition 3.8 that there exist $M,\tau\in L^{0}_{+}({\mathcal F})$ and $M\geq1$ such that $\|T(t)\|\leq Me^{\tau t}$.
Set $\widetilde{T}(t)=e^{-\tau t}T(t)$ for any $t\geq 0$, then it is easy to check that $\{\widetilde{T}(t):t\geq0\}$ is still a $C_{0}$--semigroup with the infinitesimal generator $A-\tau I$, and further $A$ is the infinitesimal generator of $\{T(t):t\geq0\}$ if and only if $A-\tau I$ is the infinitesimal generator of $\{\widetilde{T}(t):t\geq0\}$. Moreover, $\|\widetilde{T}(t)\|=\|e^{-\tau t}T(t)\|=e^{-\tau t}\|T(t)\|\leq M$ for any $t\geq0$. Thus, according to Lemma 4.11, it follows that $A$ is the infinitesimal generator of $\{T(t):t\geq0\}$ if and only if

$(i)^{'}$ $A-\tau I$ is closed and $D(A-\tau I)$ is dense in $S$.

$(ii)^{'}$ The resolvent set $\rho(A-\tau I)$ of $A-\tau I$ contains the set $\{\xi\in L^{0}({\mathcal F},C)~|~Re\xi\in L^{0}_{++}({\mathcal F})\}$ and for such $\xi$,
$\|R(\xi,A-\tau I)^{n}\|\leq\frac{M}{(Re\xi)^{n}}$
for any $n\in N$.

Clearly, $A-\tau I$ is closed if and only if $A$ is closed, and $D(A-\tau I)=D(A)$, thus $D(A)$ is dense in $S$.
Further, observing that $\tau\in\rho(A)$ if and only if $0\in\rho(A-\tau I)$,
hence
\begin{align}
\rho(A-\tau I)=\rho(A)-\tau:=\{\xi-\tau~|~\xi\in\rho(A)\}.
\end{align}

Next, according to $(ii)^{'}$, it follows that
\begin{align}
\{\xi\in L^{0}({\mathcal F},C)~|~Re\xi\in L^{0}_{++}({\mathcal F})\}\subset\rho(A-\tau I).
\end{align}
Combining (4.16) and (4.17) yields
$$\{\xi\in L^{0}({\mathcal F},C)~|~Re\xi\in L^{0}_{++}({\mathcal F})\}\subset\rho(A)-\tau,$$
i.e.,
$$\{\xi\in L^{0}({\mathcal F},C)~|~Re\xi\in L^{0}_{++}({\mathcal F})\}+\tau\subset\rho(A),$$
where $\{\xi\in L^{0}({\mathcal F},C)~|~Re\xi\in L^{0}_{++}({\mathcal F})\}+\tau:=\{\xi+\tau~|~\xi\in L^{0}({\mathcal F},C)~and~Re\xi\in L^{0}_{++}({\mathcal F})\}$,
i.e.,
$$\{\xi+\tau\in L^{0}({\mathcal F},C)~|~Re\xi\in L^{0}_{++}({\mathcal F})\}\subset\rho(A),$$
i.e.,
$$\{\xi\in L^{0}({\mathcal F},C)~|~Re(\xi-\tau)\in L^{0}_{++}({\mathcal F})\}\subset\rho(A),$$
thus
\begin{align}
\{\xi\in L^{0}({\mathcal F},C)~|~(Re\xi-\tau)\in L^{0}_{++}({\mathcal F})\}\subset\rho(A).
\end{align}

Conversely, if $\{\xi\in L^{0}({\mathcal F},C)~|~(Re\xi-\tau)\in L^{0}_{++}({\mathcal F})\}\subset\rho(A),$ then one can similarly obtain $\{\xi\in L^{0}({\mathcal F},C)~|~Re\xi\in L^{0}_{++}({\mathcal F})\}\subset\rho(A-\tau I)$. Thus we have
$\{\xi\in L^{0}({\mathcal F},C)~|~(Re\xi-\tau)\in L^{0}_{++}({\mathcal F})\}\subset\rho(A)$ if and only if $\{\xi\in L^{0}({\mathcal F},C)~|~Re\xi\in L^{0}_{++}({\mathcal F})\}\subset\rho(A-\tau I)$.

In the sequel, if $(ii)^{'}$ holds, it follows that (4.18) holds.
Moreover, it is clear that
$$R(\xi,A)x=\int_{0}^{+\infty}e^{-\xi s}T(s)xds$$
and
\begin{align}
R(\xi,A-\tau I)x&=\int_{0}^{+\infty}e^{-\xi s}\widetilde{T}(s)xds\nonumber\\
&=\int_{0}^{+\infty}e^{-\xi s}e^{-\tau s}T(s)xds\nonumber\\
&=\int_{0}^{+\infty}e^{-(\xi+\tau)s}T(s)xds\nonumber\\
&=R(\xi+\tau,A)x\nonumber
\end{align}
for any $x\in S$, $\xi\in L^{0}({\mathcal F},C)$ and $Re\xi\in L^{0}_{++}({\mathcal F})$. Thus, due to $(ii)^{'}$ and (4.18), we have
\begin{align}
\|R(\xi+\tau,A)^{n}x\|&=\|R(\xi,A-\tau I)^{n}x\|\leq\frac{M}{(Re\xi)^{n}}\|x\|,\nonumber
\end{align}
let $\eta=\xi+\tau$, then $\eta\in\rho(A)$ by (4.18) and further $$Re\xi=Re(\eta-\tau)=Re\eta-\tau.$$ Thus, for any $\eta\in L^{0}({\mathcal F},C)$ and $(Re\eta-\tau)\in L^{0}_{++}({\mathcal F})$, it follows that
$$\|R(\eta,A)^{n}x\|\leq\frac{M}{(Re\eta-\tau)^{n}}\|x\|.$$
Consequently,
$$\|R(\xi,A)^{n}\|\leq\frac{M}{(Re\xi-\tau)^{n}}$$
for any $\xi\in L^{0}({\mathcal F},C)$, $(Re\xi-\tau)\in L^{0}_{++}({\mathcal F})$ and $n\in N$.

Conversely, if $$\|R(\xi,A)^{n}\|\leq\frac{M}{(Re\xi-\tau)^{n}}$$
for any $\xi\in L^{0}({\mathcal F},C)$, $(Re\xi-\tau)\in L^{0}_{++}({\mathcal F})$ and $n\in N$,
let $\eta=\xi-\tau$, then $Re\eta=Re\xi-\tau$ and
\begin{align}
\|R(\eta, A-\tau I)^{n}\|=\|R(\eta+\tau, A)^{n}\|
\leq\frac{M}{(Re\eta)^{n}},\nonumber
\end{align}
which implies that \begin{align}
\|R(\xi, A-\tau I)^{n}\|
\leq\frac{M}{(Re\xi)^{n}}\nonumber
\end{align}
for any $\xi\in L^{0}({\mathcal F},C)$, $Re\xi\in L^{0}_{++}({\mathcal F})$ and $n\in N$.

$(a)\Leftrightarrow (b).$ According to Lemma 4.8, it can be similarly proved.

This completes the proof of Theorem 4.7.
\end{proof}

\begin{remark} According to Example 3.12 and Remark 3.13, one can conclude that it is necessary to require the almost sure boundedness for such $C_{0}$--semigroups in Theorem 4.7 (a), which is quite different from the classical case.
\end{remark}

If we choose ${\mathcal F}=\{\Omega,\Phi\}$, then a complete $RN$ module $S$ reduces to a Banach space $X$ and the a.s. bounded $C_{0}$--semigroup $\{T(t):t\geq0\}$ reduces to an ordinary $C_{0}$--semigroup on $X$, which leads to the following Corollary 4.13.

\begin{corollary}(see \cite{Engel2000}) Let $(A,D(A))$ be a linear operator on a Banach space $X$. Suppose that $\tau\in R$, $M \in R^+$ and $M\geq1$. Then the following assertions are equivalent.

(a) $(A,D(A))$ generates a $C_{0}$--semigroup $\{T(t):t\geq0\}$ satisfying
$\|T(t)\|\leq Me^{\tau t}$
for any $t\geq0$.

(b) $(A,D(A))$ is closed, densely defined, and for each $\xi>\tau$ one has $\xi\in \rho(A)$ and
$\|[(\xi-\tau)R(\xi,A)]^{n}\|\leq M$
for any $n\in N$.

(c) $(A,D(A))$ is closed, densely defined, and for each $\xi\in C$ with $Re\xi>\tau$ one has $\xi\in \rho(A)$ and
$\|R(\xi,A)^{n}\|\leq\frac{M}{(Re\xi-\tau)^{n}}$
for any $n\in N$.

\end{corollary}

\begin{remark} It is worth mentioning that Theorem 4.7 is a generalization of Corollary 4.13 since a continuous function from a finite closed interval $[a, b]$ to a Banach space $X$ is naturally bounded.
\end{remark}

\section*{Acknowledgment}
This paper is supported by the Natural Science
Foundation of Tianjin (Grant No. 18JCYBJC18900), the Humanities and Social Science Foundation of Ministry of Education (Grant No. 20YJC790174) and the NNSF of China (Grant Nos. 11301380 and 11571369).


\end{document}